\documentclass[final,letterpaper,openacc]{PRM-mod}

\usepackage[british]{babel}
\usepackage{microtype} 
\usepackage[babel,autostyle]{csquotes} 

\usepackage{amssymb,amsmath,amsthm}
\usepackage{tikz-cd}

\newcommand{\semi}{\rtimes}
\newcommand{\iso}{\cong}
\newcommand{\normal}{\vartriangleleft}
\renewcommand{\phi}{\varphi}
\providecommand{\abs}[1]{\ensuremath{\left\lvert#1\right\rvert}}

\DeclareMathOperator{\inv}{inv}
\DeclareMathOperator{\res}{res}
\DeclareMathOperator{\Syl}{Syl}

\usepackage{natbib}
\def\newblock{\ }

\newtheorem{theorem}{Theorem} 
\newtheorem{lemma}{Lemma}   
\newtheorem{corollary}{Corollary}

\newtheorem{proposition}{Proposition}

\begin{document}

\title{Fixed point conditions for non-coprime actions}

\author{\name{Michael C. \surname{Burkhart}}}
\address{University of Cambridge, Cambridge, United Kingdom
	\email{mcb93@cam.ac.uk}}

\begin{abstract}
	In the setting of finite groups, suppose $J$ acts on $N$ via automorphisms
	so that the induced semidirect product $N\rtimes J$ acts on some non-empty
	set $\Omega$, with $N$ acting transitively. Glauberman proved that if the
	orders of $J$ and $N$ are coprime, then $J$ fixes a point in $\Omega$. We
	consider the non-coprime case and show that if $N$ is abelian and a Sylow
	$p$-subgroup of $J$ fixes a point in $\Omega$ for each prime $p$, then $J$
	fixes a point in $\Omega$. We also show that if $N$ is nilpotent, $N\semi
	J$ is supersoluble, and a Sylow $p$-subgroup of $J$ fixes a point in
	$\Omega$ for each prime $p$, then $J$ fixes a point in $\Omega$.
\end{abstract}



\classification[\textup{2020} Mathematics subject classification]{05E18; 20E22;
	20E45; 20F16; 20J06; 55N45}

\maketitle
\section{Introduction}
\label{s:intro}
Suppose a finite group $J$ acts via automorphisms on a finite group $N$ and the
induced semi-direct product $G= N \semi J$ acts on some non-empty set $\Omega$
where the action of $N$ is transitive. Glauberman showed that if each
supplement $H$ of $N$ in $G$ splits over $N \cap H$ and each complement of $N$
in $G$ is conjugate to $J$, then there exists a $J$-invariant element
$\omega\in\Omega$. Consequently, if the orders of $J$ and $N$ are coprime so
that the Schur--Zassenhaus theorem applies, a fixed point always
exists~\cite[Thm. 4]{Gla64}. In this note, we consider the non-coprime case and
establish some conditions for the existence of a fixed point.

Given an action as described above, consider the stabiliser $G_\alpha\leq G$
fixing an arbitrary point $\alpha\in\Omega$. As $N$ is transitive, $G_\alpha$
supplements $N$ in $G$. In this context, $J$ fixes an element of $\Omega$ if
and only if the following two conditions are met. Firstly, we must ensure
$G_\alpha$ splits over $N\cap G_\alpha$ so that there exists some complement
$J'$. As $G/N \iso G_\alpha/(N\cap G_\alpha)$, it will follow that $J'$ also
complements $N$ in $G$. Secondly, we require that $J' = g^{-1}Jg$ for some
$g\in G$ so that $J$ fixes $g\cdot \alpha$. For the latter requirement, we
concern ourselves with conditions for two specific complements in a semidirect
product to be conjugate.

To this end, we say two subgroups $H$ and $H'$ are \emph{locally conjugate} in
a group $G$ if for each prime $p$, a Sylow $p$-subgroup of $H$ is conjugate to
a Sylow $p$-subgroup of $H'$. Losey and Stonehewer showed that if $H$ and $H'$
are locally conjugate supplements of some normal nilpotent subgroup $N$ in a
soluble group $G$, then $H$ and $H'$ are conjugate if either $G/N$ is nilpotent
or $N$ is abelian~\cite{Los79}. Evans and Shin further showed that if $N$ is
abelian, then $G$ need not be soluble~\cite{Eva88}.

\newpage

We first restrict $N$ to be abelian and use a decomposition result from group
cohomology to provide an alternate proof of:
\begin{lemma}[Evans and Shin]
	\label{lem:ab}
	In a finite group, two complements of a normal abelian subgroup are
	conjugate if and only if they are locally conjugate.
\end{lemma}
We use this, along with Gasch\"utz’s result that a finite group $G$ splits over
an abelian subgroup $N$ if and only if for each prime $p$, a Sylow $p$-subgroup
$S$ of $G$ splits over $N\cap S$, to show:
\begin{theorem} 
	\label{thm:ab} 
	Given a finite group $J$ acting via automorphisms on a finite abelian group
	$N$, suppose the induced semidirect product $N\semi J$ acts on some
	non-empty set $\Omega$ where the action of $N$ is transitive. If for each
	prime $p$, a Sylow $p$-subgroup of $J$ fixes an element of $\Omega$, then
	there exists some $J$-invariant element $\omega\in\Omega$.
\end{theorem}
This had previously been shown using elementary arguments for the special case
that $J$ is supersoluble~\cite[Cor. 2]{Bur22}. The theorem implies:
\begin{corollary}
	\label{cor:ab}
	Let $G$ be a finite split extension over an abelian subgroup $N$. If for
	each prime $p$ there is a Sylow $p$-subgroup $S$ of $G$ such that any two
	complements of $N\cap S$ in $S$ are conjugate, then any two complements of
	$N$ in $G$ are $G$-conjugate.
\end{corollary}
This extends a result of D.\,G. Higman~\cite[Cor. 2]{Hig54} that requires the
complements of $N\cap S$ in $S$ to be conjugate \emph{within $S$}.

We then consider nilpotent $N$ and supersoluble $N\semi J$. We adapt our
approach for Lemma~\ref{lem:ab} to nonabelian cohomology and demonstrate:
\begin{lemma}
	\label{lem:nil}
	In a finite supersoluble group, two complements of a normal nilpotent
	subgroup are conjugate if and only if they are locally conjugate.
\end{lemma}
With this, we then show:
\begin{theorem} 
	\label{thm:nil} 
	Given a finite group $J$ acting via automorphisms on a finite nilpotent
	group $N$, suppose the induced semidirect product $N\semi J$ is
	supersoluble and acts on some non-empty set $\Omega$ where the action of
	$N$ is transitive. If for each prime $p$, a Sylow $p$-subgroup of $J$ fixes
	an element of $\Omega$, then there exists some $J$-invariant element
	$\omega\in\Omega$.
\end{theorem}
The theorem also implies an analogue of Corollary~\ref{cor:ab} that we state
and prove in \S~\ref{s:nil}.

\subsection{Outline}
We proceed as follows.  In the remainder of this section, we introduce notation
and some conventions from group cohomology. In the next section, we restrict
$N$ to be abelian and prove Theorem~\ref{thm:ab}. We then restrict $N$ to be
nilpotent and $N\semi J$ to be supersoluble in \S~\ref{s:nil} and prove
Theorem~\ref{thm:nil}, before concluding in \S~\ref{s:conc}.

\subsection{Notation and conventions}
All groups in this note are assumed finite. A subgroup $K\leq G$ supplements
$N\normal G$ if $G=NK$ and complements $N$ if it both supplements $N$ and the
intersection $N\cap K$ is trivial. We denote conjugation by $g^\gamma =
	\gamma^{-1}g\gamma$ for $g,\gamma\in G$ and otherwise let groups act from the
left. For a prime $p$, we let $\Syl_p(G)$ denote the set of Sylow $p$-subgroups
of a group $G$.

We rely on rudimentary notions from group cohomology that can be found in the
texts of Brown~\cite{Bro82} and Serre~\cite{Ser02}. Given a group $J$ acting on
a group $N$ via automorphisms, crossed homomorphisms or 1-cocycles are maps
$\phi: J \to N$ satisfying $\phi(jj') = \phi(j) \phi(j')^{j^{-1}}$ for all
$j,j' \in J$. Two such maps $\phi$ and $\phi'$ are cohomologous if there exists
$n\in N$ such that $\phi'(j) =  n^{-1} \phi(j) n^{j^{-1}}$ for all $j\in J$; in
this case, we write $\phi\sim\phi'$. We take the first cohomology $H^1(J,N)$ to
be the pointed set $Z^1(J,N)$ of crossed homomorphisms modulo this equivalence.
The distinguished point corresponds to the equivalence class containing the map
taking each element of $J$ to the identity of $N$. Our interest in this set
stems primarily from the well-known bijective correspondence~\cite[Exer. 1 in
\S I.5.1]{Ser02} between it and the $N$-conjugacy classes of complements to $N$
in $N\semi J$. Specifically, for each $\phi \in Z^1(J,N)$, the subgroup
$F(\phi) = \{ \phi(j)j \}_{j\in J}$ complements $N$ in $NJ$ and all such
complements may be written in this way. Two crossed homomorphisms yield
conjugate complements under $F$ if and only if they are cohomologous, so $F$
induces the desired correspondence.

For a subgroup $K\leq J$, we let $\phi|_{K}$ denote the restriction of $\phi\in
	Z^1(J,N)$ to $K$ and $\res^J_K: H^1(J,N) \to H^1(K,N)$ be the map induced in
cohomology. For $\phi \in Z^1(K,N)$ and $j\in J$, define $\phi^j(x) =
	\phi(x^{j^{-1}})^{j}$. We call $\phi$ $J$-invariant if $\res^K_{K\cap K^j} \phi
	\sim \res^{K^j}_{K\cap K^j} \phi^j$ for all $j\in J$ and let $\inv_J H^1(K,N)$
denote the set of $J$-invariant elements in $H^1(K,N)$. For any $\phi \in
	Z^1(J,N)$, we have $\phi^j(x) = n^{-1} \phi(x) n^{x^{-1}}$ where $n=
	\phi(j^{-1})$ so that $\phi^j \sim \phi$. In particular, $\res^J_K H^1(J,N)
	\subseteq \inv_J H^1(K, N)$.

\section{$N$ is abelian}
\label{s:ab}

In this section, we restrict $N$ to be abelian so that $H^1(J,N)$ takes the
form of an abelian group. We first prove Lemma~\ref{lem:ab} as stated in
\S~\ref{s:intro}.

\begin{proof}[Proof of Lemma~\ref{lem:ab}] \quad 
	Suppose we are given locally conjugate complements $J$ and $J'$ of a normal
	abelian subgroup $N$ in some group $G$. As any element $g\in G$ may be
	uniquely written $g=jn$ for $j\in J$ and $n\in N$, for each prime $p$ we have
	$J_p' = (J_p)^n$ for some $J_p\in \Syl_p(J)$, $J_p' \in\Syl_p(J')$, and
	$n\in N$. Let $\phi'\in Z^1(J,N)$ denote the crossed homomorphism
	corresponding to $J'$. It suffices to show that $\phi'\sim 1$, where $1\in
	Z^1(J,N)$ denotes the map taking each element of $J$ to the identity of
	$N$. Through the $p$-primary decomposition of $H^1(J,N)$, we have the
	isomorphism~\cite[\S III.10]{Bro82}:
	\begin{equation}
		\label{eq:decomp}
		H^1(J, N) \iso \oplus_{p\in\mathcal{D}} \inv_J H^1(J_p, N)
	\end{equation}
	where $\mathcal{D}$ is the set of prime divisors of $\abs{J}$ and the $J_p$
	are those given above. For every $p\in \mathcal D$, we see that
	$\phi'|_{J_p}\sim 1|_{J_p}$ as $J_p$ and $J_p'$ are $N$-conjugate
	complements of $N$ in $NJ_p$. Thus, $\phi'$ maps to the identity in each
	direct summand on the right hand side of~\eqref{eq:decomp} and we may
	conclude $\phi'\sim 1$ so that $J$ and $J'$ are conjugate.
\end{proof}

We can now use the lemma and Gasch\"utz's theorem to prove Theorem~\ref{thm:ab}.
\begin{proof}[Proof of Theorem~\ref{thm:ab}] \quad 
	Given $J$, $N$, and $\Omega$ as described in the hypotheses of the theorem,
	let $G= N \semi J$ denote the induced semidirect product and consider the
	stabiliser subgroup $G_\alpha$ for some fixed $\alpha\in\Omega$. As $N$
	acts transitively, any $g\in G$ may be written $g\cdot\alpha = n\cdot
	\alpha$ for some $n\in N$, so that $n^{-1}g \in G_\alpha$. Thus,
	$G=NG_\alpha$.

	We claim $G_\alpha$ splits over $N\cap G_\alpha$. For any prime $p$, there
	exists by hypothesis some $n\in N$ and $P \in \Syl_p(J)$ such that $P^n
		\leq G_\alpha$. Let $L\in \Syl_p(N\cap G_\alpha)$. As $\abs{G_\alpha} =
		\abs{N\cap G_\alpha}[G:N]$, it follows that $S=LP^n \in \Syl_p(G_\alpha)$
	so $P^n$ complements $S \cap N = L$ in $S$. As the choice of prime $p$ was
	arbitrary, we may apply Gasch\"utz's theorem to conclude that $G_\alpha$
	splits over $N\cap G_\alpha$.

	Let $J'$ complement $N\cap G_\alpha$ in $G_\alpha$. As $G/N \iso G_\alpha /
		(N \cap G_\alpha)$, it follows that $J'$ also complements $N$ in $G$.
	Lemma~\ref{lem:ab} then implies that $J' = J^g$ for some $g\in G$ so
	that $J$ fixes $\omega= g\cdot \alpha$.
\end{proof}

Finally, we outline how Corollary~\ref{cor:ab} follows from
Theorem~\ref{thm:ab}.
\begin{proof}[Proof of Corollary~\ref{thm:ab}] \quad 
	Given a group $G$ satisfying the hypotheses of the corollary, suppose $J$
	and $J'$ each complement $N$ in $G$. Then $G$ acts on the cosets $\Omega =
	G/J'$ in such a way that we may apply Theorem~\ref{thm:ab} to infer that
	$J$ fixes $gJ'$ for some $g\in G$. Therefore, $J$ and $J'$ are conjugate.
	As the choice of complements was arbitrary, we may conclude.
\end{proof}

\section{$N$ is nilpotent and $N\semi J$ is supersoluble}
\label{s:nil}

In this section, we suppose that $N$ is nilpotent and $N \semi J$ is
supersoluble. Consequently, $N$ decomposes as the direct sum $N\iso
	\oplus_{p\in\mathcal D} N_p$ over its characteristic Sylow $p$-subgroups $N_p$
where $\mathcal D$ denotes the set of prime divisors of $\abs{N}$. Direct
calculations show that the natural projections $N \to N_p$ induce an
isomorphism of pointed sets
\begin{equation}
	\label{eq:prod}
	H^1(J,N) \iso \oplus_{p\in D} H^1(J,N_p).
\end{equation}
To parse the components on the right hand side of~\eqref{eq:prod}, we introduce
the following:

\begin{proposition}
	\label{prop:nilp}
	Suppose a group $J$ acts on a $p$-group $N$ via automorphisms, so that the
	induced semidirect product $N\semi J$ is supersoluble. Then $\res_{J_p}^J:
	H^1(J,N) \to \inv_J H^1(J_p,N)$ is an isomorphism for $J_p\in\Syl_p(J)$.
\end{proposition}

\begin{proof} \quad 
	We induct on the order of $J$. If $J$ itself is a $p$-group, the conclusion
	is immediate. If $p$ is not a divisor of $\abs{J}$, the lemma follows from
	the Schur--Zassenhaus theorem. Otherwise, let $Q\normal J$ be a Sylow
	$q$-subgroup where $q$ is the largest prime divisor of
	$\abs{J}$~\cite[exer. 3B.10]{Isa08} so that $J\iso Q \semi M$ for some Hall
	$q'$-subgroup $M\leq J$. Consider the inflation-restriction exact
	sequence~\cite[\S I.5.8]{Ser02},
	\begin{equation} \label{eq:inf_res}
		1 \to H^1(J/Q, N^Q) \to H^1(J,N) \xrightarrow{\res^J_Q} H^1(Q,N)^{J/Q}
	\end{equation}
	where $N^Q$ denotes the elements of $N$ fixed by $Q$.

	If $q\ne p$, then $H^1(Q,N)$ is trivial so that $ H^1(J,N) \iso H^1(M,
	N^Q)$. In the supersoluble group $NQ$, $Q$ is a Sylow $q$-subgroup for the
	largest prime divisor of $\abs{NQ}$, so that $Q\normal N Q$ and $N^Q = N$.
	Consequently, $H^1(J,N) \iso H^1(M, N)$. We claim that $\res^J_M$ affords
	this isomorphism. It suffices to show that $\res^J_M$ is surjective. For
	any $\phi\in Z^1(M, N)$, we may define $\tilde\phi: J \to N$ by
	$\tilde\phi(qm) = \phi(m)$ for $q\in Q$ and $m\in M$. This map is
	well-defined as $J\iso Q \semi M$. For $q,q'\in Q$ and $m,m'\in M$, we have
	$\tilde\phi(qmq'm') = \phi(mm') = \phi(m)\phi(m')^{m^{-1}} =
	\tilde\phi(qm)\tilde\phi(q'm')^{(qm)^{-1}}$, where the last equality
	follows from the fact that elements of $N$ commute with elements of $Q$.
	Thus, $\tilde\phi \in Z^1(J, N)$. As $\tilde\phi|_M = \phi$, we conclude
	$\res^J_M$ is surjective.
	
	Exchanging $M$ for a conjugate if
	necessary, we may assume that $J_p \leq M$. As $\res^M_{J_p}$ is injective by
	induction, it follows that the composition $\res^J_{J_p} = \res^M_{J_p} \circ
		\res^J_M$ is also injective. On the other hand,
	\[
		\inv_J H^1(J_p, N) \subseteq \inv_{M} H^1(J_p,N)
		= \res^M_{J_p} H^1(M,N) \subseteq  \res^J_{J_p} H^1(J,N)
	\]
	where the equality above follows from the inductive hypothesis, so that
	$\res^J_{J_p}$ is surjective.

	Otherwise, $q=p$, so that  $J_p=Q$ is a Sylow $p$-subgroup of $J$. In this
	case, $H^1(J/Q, N^Q)$ is trivial in~\eqref{eq:inf_res} and so
	$\res^J_{J_p}$ is injective. As $H^1(Q,N)^{J/Q} = \inv_J H^1(Q,N)$, it
	remains to show that this map is surjective. For $M$-invariant $\phi \in
	Z^1(J_p, N)$, define $\tilde \phi: J \to N$ by $\tilde \phi(hm) = \phi(h)$
	for $h\in J_p$ and $m\in M$. Then for any $h,h' \in J_p$ and $m, m' \in M$,
	we have $\tilde \phi(hmh'm') = \phi(h (h')^{m^{-1}}) = \phi(h)
	\phi((h')^{m^{-1}})^{h^{-1}} = \phi(h) \phi(h')^{m^{-1}h^{-1}} = \tilde
	\phi(hm) \tilde \phi(h'm')^{(hm)^{-1}}$ where the third equality follows
	from $\phi$ being $M$-invariant. As $J\iso J_p \semi M$, we conclude that
	$\tilde \phi \in Z^1(J, N)$. Clearly, $\res^J_{J_p} \tilde \phi \sim \phi$
	so that $\res^J_{J_p}$ is surjective.
\end{proof}

For each prime $p$, we may apply Proposition~\ref{prop:nilp} to the component
for $p$ in~\eqref{eq:prod} and find that $H^1(J,N_p)\iso \inv_J H^1(J_p, N_p)
\iso \inv_J H^1(J_p, N)$ for some $J_p \in \Syl_p(J)$. In particular, it
follows that:

\begin{proposition}
	\label{prop:nil}
	Given a group $J$ acting on a nilpotent group $N$ via automorphisms so that
	$N\semi J$ is supersoluble, the restriction maps $\res^J_{J_p}$ induce an
	isomorphism of pointed sets $H^1(J, N) \iso \oplus_{p\in\mathcal{D}} \inv_J
	H^1(J_p, N)$ where $\mathcal{D}$ denotes the set of prime divisors of
	$\abs{J}$ and $J_p\in\Syl_p(J)$ for each $p\in\mathcal D$.
\end{proposition}

We are now prepared to provide a proof of Lemma~\ref{lem:nil}.
\begin{proof}[Proof of Lemma~\ref{lem:nil}] \quad 
	In a supersoluble group $G$, suppose $J$ and $J'$ are locally conjugate
	complements of a normal nilpotent subgroup $N$. As in Lemma~\ref{lem:ab},
	we have for each prime $p$ that some $J_p\in\Syl_p(J)$ and
	$J_p'\in\Syl_p(J')$ are conjugate by an element of $N$. Let $\phi'\in
	Z^1(J,N)$ denote the map corresponding to $J'$. As the isomorphism in
	Proposition~\ref{prop:nil} is induced by restriction maps, it takes the
	identity $1\in H^1(J, N)$ to $\oplus_{p\in\mathcal{D}} 1|_{J_p}$. Thus, as
	$\phi'|_{J_p}\sim1|_{J_p}$ for each $p\in\mathcal D$, we may apply
	Proposition~\ref{prop:nil} to conclude $\phi'\sim 1$ so that $J$ and $J'$
	are conjugate.
\end{proof}

We now use Lemma~\ref{lem:nil} to show:

\begin{proposition}
	\label{prop:nil_split}
	Let $H$ be a subgroup of some supersoluble $G\iso N \semi J$ where $N$ is
	nilpotent. If for each prime $p$, $H$ contains a conjugate of some $S\in
	\Syl_p(J)$, then $H$ contains a conjugate of $J$ and so splits over $N\cap
	H$.
\end{proposition}

\begin{proof} \quad 
	The hypotheses imply that $H$ supplements $N$ in $G$. We induct on the
	order of $G$. If $N$ is trivial or if $H$ is a $p$-group, the conclusion
	follows immediately. If multiple primes divide $\abs{N}$, then for some
	prime $p$, $HN_p$ must be a strict subgroup of $G$ for $N_p \in \Syl_p(N)$;
	otherwise $H$ would contain a Sylow subgroup of $G$ for each prime and we
	would have $H=G$.  Let $p$ be such a prime. Induction in $G/N_p$
	implies $J^g \leq HN_p$ for some $g\in G$. Switching to a conjugate of $H$
	if necessary, we may assume that $g$ is trivial and apply the inductive
	hypothesis in $HN_p$ to conclude $J^{g'} \leq H$ for some $g'\in G$. We now
	proceed under the assumption that $N$ is a $q$-subgroup for some prime $q$. 
	
	Let $A \leq N$ be a minimal normal subgroup of $G$; as $G$ is supersoluble,
	it will have prime order $q$. If $A \leq H$, then in $G/A$, induction
	implies that $J^gA \leq HA = H$ for some $g\in G$ so that $J^g \leq H$.
	
	Otherwise, $A \cap H$ is trivial. Without loss, $J_q \leq H$ for some $J_q
	\in \Syl_q(J)$. In $G/A$, induction implies that a conjugate of $JA/A$ is
	contained in $HA/A$. Let $\overline{K}$ denote this conjugate. Switching to
	a different conjugate if necessary, we may assume that $J_qA/A \leq
	\overline{K}$. Let $\phi: h \mapsto hA/A$ denote the isomorphism from $H$
	to $HA/A$ and consider $K = \phi^{-1}(\overline K)$. It follows that $J_q
	\leq K$ and $\abs{K} = \abs{J}$ so that $K\leq H$ complements $N$ in $G$.
	As $N$ is a $q$-group, a Sylow $p$-subgroup of $J$ will be conjugate to a
	Sylow $p$-subgroup of $K$ for primes $p\ne q$. Lemma~\ref{lem:nil} then
	implies that $J$ and $K\leq H$ are conjugate in $G$.
\end{proof}

We now prove Theorem~\ref{thm:nil}.

\begin{proof}[Proof of Theorem~\ref{thm:nil}] \quad 
	Given $J$, $N$, and $\Omega$ as described in the hypotheses of the theorem,
	let $G= N \semi J$ denote the induced semidirect product and consider
	$G_\alpha$ for some $\alpha\in\Omega$. As $N$ acts transitively,
	$G=NG_\alpha$. For each prime $p$, the hypotheses of the theorem imply
	$(J_p)^{n_p} \leq G_\alpha$ for some $J_p \in\Syl_p(J)$ and $n_p\in N$, so
	that Proposition~\ref{prop:nil_split} implies $G_\alpha$ contains a
	conjugate of $J$, say $J^g$ for $g\in G$. It follows that $J$ fixes
	$\omega=g \cdot a$.
\end{proof}

This in turn implies:
\begin{corollary}
	\label{cor:nil}
	Let $G$ be a supersoluble split extension over a nilpotent subgroup $N$. If
	for each prime $p$ there is a Sylow $p$-subgroup $S$ of $G$ such that any
	two complements of $S\cap N$ in $S$ are conjugate, then any two complements
	of $N$ in $G$ are conjugate.
\end{corollary}

\begin{proof} \quad 
	Suppose arbitrary $J$ and $J'$ complement $N$ in $G$. Then $G$ acts on the
	cosets $\Omega = G/J'$ in such a way that we may apply
	Theorem~\ref{thm:nil} to infer that $J$ fixes $gJ'$ for some $g\in G$.
	Consequently, $J$ and $J'$ are conjugate, and we may conclude.
\end{proof}

\section{Concluding remarks}
\label{s:conc}

In their paper, Losey and Stonehewer exhibited a soluble group $G\iso N \semi
	J$ with $N$ nilpotent and $J$ supersoluble and a second complement $J'$ to $N$
in $G$ such that $J$ and $J'$ are locally conjugate but not
conjugate~\cite{Los79}. Thus, Lemma~\ref{lem:nil} cannot be extended to
supersoluble complements of a normal nilpotent subgroup in a soluble group.

\section*{Acknowledgments}
The author thanks Elizabeth Crites, the editor Alex Bartel, and an anonymous
reviewer for thoughtful and detailed feedback on the manuscript.


\end{document}